\documentclass[10pt]{article}
\usepackage{amsmath,amsthm,amsfonts,amssymb,url,authblk}
\newtheorem{theorem}{Theorem}[section]
\newtheorem{lemma}[theorem]{Lemma}
\newtheorem{corollary}[theorem]{Corollary}
\newtheorem{remark}{Remark}[section]

\newtheorem{example}{Example}[section]

\newcommand{\OA}{\ensuremath{\mathsf{OA}}}

\newcommand{\zed}{\ensuremath{\mathbb{Z}}} 
\newcommand{\C}{\ensuremath{\mathcal{C}}}

\title{Constructions of optimal orthogonal arrays with repeated rows}
\author{Charles J.\ Colbourn}
\affil{School of Computing, Informatics and Decision Systems Engineering,
Arizona State University,
Tempe, Arizona 85287,
U.S.A.}
\author{Douglas R.\ Stinson%
\thanks{D.R.\ Stinson's research is supported by  NSERC discovery grant RGPIN-03882.}}
\affil{David R.\ Cheriton School of Computer Science, University of Waterloo,
Waterloo, Ontario N2L 3G1, Canada}
\author{Shannon Veitch}
\affil{Dept.\ of Combinatorics and Optimization, University of Waterloo,
Waterloo, Ontario N2L 3G1, Canada}

\date{\today}

\begin{document}
\maketitle

\begin{abstract}
We construct  orthogonal arrays  $\OA_{\lambda} (k,n)$ (of strength two)
having a row that is repeated $m$ times, where $m$ is as large
as possible. In particular, we consider OAs where the ratio $m / \lambda$ is
as large as possible; these OAs are termed \emph{optimal}. We provide constructions of
optimal OAs for any $k \geq n+1$, albeit with large $\lambda$. We also study
\emph{basic} OAs; these are optimal OAs in which $\gcd(m,\lambda) = 1$. We  construct
a basic OA with $n=2$ and $k =4t+1$, provided that a  Hadamard matrix of order $8t+4$ exists.
This completely solves the problem of constructing basic OAs wth $n=2$, modulo the 
Hadamard matrix conjecture. 
\end{abstract}

\section{Introduction}
\label{OA2.sec}

Let $k \geq 2$, $n \geq 2$ and $\lambda \geq 1$ be integers.
An \emph{orthogonal array} $\OA_{\lambda} (k,n)$
is a $\lambda  n^2 $ by $k$ array, $A$, with entries from a set
$X$ of cardinality $n$ such that, within any two columns of $A$,
every ordered pair of symbols from $X$ occurs in exactly $\lambda$ rows of $A$.
%We often denote the number of rows, $\lambda  n^2$, by $N$.
For much information on orthogonal arrays, see \cite{HSS}.

In this note, we are interested in $\OA_{\lambda} (k,n)$ that contain a row
that is repeated $m$ times, where $m$ is as large as possible. 
 We observe that,
by relabelling the symbols in the orthogonal array, these $m$  rows can all be assumed
to be of the form $x \; x \cdots x$ for some fixed symbol $x$. We
will denote this particular symbol $x$ either by  $0$ or by $\infty$.

The following theorem, along with an elementary combinatorial proof, can be found in 
\cite[Theorem 2.2]{Sti18}. However, we should note that this result is also
an immediate consequence of a much more general result from \cite{Muk}.

\begin{theorem}
\label{repeated.thm}
Let $k \geq 2$, $n \geq 2$ and $\lambda \geq 1$ be integers.
If there is an $\OA_{\lambda}(k,n)$ containing a row that is repeated $m$ times, then
\[
m \leq \frac{\lambda n^2}{k(n-1)+1}.
\]
\end{theorem}

An $\OA_{\lambda}(k,n)$, say $A$, containing a row 
that is repeated 
\begin{equation}
\label{m.eq}
m = \frac{\lambda n^2}{k(n-1)+1} 
\end{equation}
times will be termed \emph{optimal}. Another way to view the optimality property is
to observe that the ratio $m /\lambda$ is as large as possible in an optimal OA.

We note that, in a recent paper, Culus and  Toulouse \cite{CT} discuss an application where it is 
beneficial to construct optimal orthogonal arrays. They also construct several small examples of optimal OAs
using linear programs.

The rest of this paper is organized as follows. Section \ref{prelim.sec} establishes some basic results
about optimal OAs. In Section \ref{examples.sec}, we present some small examples of basic OAs, which are optimal 
OAs in which $\gcd(m,\lambda) = 1$. 
In Section \ref{basic.sec}, we give a 
complete solution (modulo the Hadamard matrix conjecture) to the the problem of constructing basic OAs with $n = 2$.
Section \ref{general.sec} gives constructions of optimal OAs for arbitrary values of $n$ and $k \geq n+1$.
Section \ref{delete.sec} examines the effect of deleting a small number of columns from an optimal OA.
Finally, Section \ref{summary.sec} is a brief summary.

\section{Preliminary results}
\label{prelim.sec}

From the proof of \cite[Theorem 2.2]{Sti18}, the following result can be 
derived immediately.

\begin{lemma}
\label{abar.lem}
\cite[Corollary 2.4]{Sti18}
Let $k \geq 2$, $n \geq 2$ and $\lambda \geq 1$ be integers.
Suppose there is an optimal $\OA_{\lambda}(k,n)$, say $A$, containing a row 
$0 \; 0 \cdots 0$ that is repeated $m$ times.
Then every other row of $A$ contains exactly 
\begin{equation}
\label{a.eq}
\overline{a} = \frac{k(\lambda n - m)}{\lambda n^2-m}
\end{equation} 
occurrences of the symbol $0$.
\end{lemma}

We can  simplify the formula (\ref{a.eq}) for $\overline{a}$, as follows.

\begin{corollary}
\label{a.cor}
Let $k \geq 2$, $n \geq 2$ and $\lambda \geq 1$ be integers.
Suppose there is an optimal $\OA_{\lambda}(k,n)$, say $A$, containing a row 
$0 \; 0 \cdots 0$ that is repeated $m$ times.
Then every other row of $A$ contains exactly $(k-1)/n$
%\begin{equation}
%\label{a2.eq}
%\overline{a} = \frac{k-1}{n}
%\end{equation} 
occurrences of the symbol $0$ and thus $k \equiv 1 \bmod n$.
\end{corollary}

\begin{proof}
In an optimal $\OA_{\lambda}(k,n)$, equation (\ref{m.eq}) holds.
Suppose we substitute this expression for $m$ into equation (\ref{a.eq}).
We obtain
{\allowdisplaybreaks 
\begin{eqnarray*}
\overline{a} &=& \frac{k(\lambda n - m)}{\lambda n^2-m}\\
&=& \frac{k\left(\lambda n - \frac{\lambda n^2}{k(n-1)+1}\right)}{\lambda n^2-\frac{\lambda n^2}{k(n-1)+1}}\\
&=& \frac{k\left(1 - \frac{n}{k(n-1)+1}\right)}{n-\frac{n}{k(n-1)+1}}\\
&=& \frac{k(k(n-1)+1) -kn}{n(k(n-1)+1) - n}\\
&=& \frac{k^2n - k^2 +k-kn}{kn(n-1)}\\
&=& \frac{k(k-1)(n-1)}{kn(n-1)}\\
&=& \frac{k-1}{n}.
\end{eqnarray*}
}
\end{proof}

We will define a quadruple of positive integers $(m,\lambda,k,n)$  to be \emph{feasible}
if the following conditions are satisfied:
\begin{itemize}
\item $k \geq 2$ and $n \geq 2$,
\item $m = \frac{\lambda n^2}{k(n-1)+1}$, and
\item $\overline{a} = \frac{k-1}{n}$ is a positive  integer.
\end{itemize}
The above conditions are all necessary for the existence of an optimal
$\OA_{\lambda}(k,n)$.

\begin{lemma}
Suppose $(m,\lambda,k,n)$  is a feasible quadruple, $\ell \mid m$ and  $\ell \mid \lambda$,
where $\ell > 1$. Then $(m/\ell,\lambda/\ell,k,n)$  is a feasible quadruple.
\end{lemma}

\begin{proof} Define $m' = m/\ell$ and $\lambda' = \lambda / \ell$. 
It suffices to observe that
\[ m = \frac{\lambda n^2}{k(n-1)+1} \Leftrightarrow m' = \frac{\lambda' n^2}{k(n-1)+1}.\]
\end{proof}

If there exists an optimal $\OA_{\lambda'}(k,n)$
and we take $\ell$ copies of every row, then we obtain an optimal
$\OA_{\lambda}(k,n)$, where $\lambda = \ell \lambda'$. The most interesting parameter cases are those where
we cannot just take multiple copies of a smaller OA.
Therefore, we define a feasible quadruple $(m,\lambda,k,n)$ to be \emph{basic} if $\gcd (m,\lambda) = 1$.
Similarly, an optimal
$\OA_{\lambda}(k,n)$  is \emph{basic} if $\gcd (m,\lambda) = 1$.

As we already mentioned, an optimal OA is any $\OA_{\lambda}(k,n)$ that has an $m$-times
repeated row, where the ratio $m/\lambda$ is as large as possible. A
basic  OA is an optimal OA where the value of $\lambda$ (or equivalently, $m$) 
is as small as possible. Note that basic OAs include the case $m = \lambda= 1$, which 
of course do not contain repeated rows. 

The possible basic quadruples are quite constrained if $n$ is prime and $m > 1$.

\begin{theorem}
\label{prime.thm}
If $(m, \lambda, k, n)$ is a basic quadruple, $m > 1$ and $n$ is prime, 
then $m = n$. Further, $k = ns+1$ and $\lambda = (n-1)s + 1$ for some integer $s$
such that $\gcd(n,s-1) = 1$.
\end{theorem}

\begin{proof}
We have that $m (k (n-1) + 1) = \lambda n^2$ from (\ref{m.eq}).
Further, $k = ns + 1$ for some integer $s$, by Corollary \ref{a.cor}. Therefore
\begin{eqnarray*}
\lambda n^2 &=& m ((ns+1)(n-1) + 1) \\
&=& m (n^2s - ns + n)\end{eqnarray*}
and hence
\[ \lambda n = m (ns -s + 1).\]
Then $m \mid \lambda n$ and $\gcd(m, \lambda) = 1$, so $m \mid n$. 
Since $n$ is prime and $m > 1$, we have $m = n$.

Referring again to (\ref{m.eq}), it follows that
\begin{eqnarray*}
\lambda &=& \frac{m (k(n-1) + 1)}{n^2} \\
&=& \frac{n  ((ns+1)(n-1) + 1)}{n^2} \\
&=& (n-1)s + 1.
\end{eqnarray*}

Finally, a basic quadruple requires that $\gcd(m,\lambda) = 1$.
Thus, it is necessary that $\gcd (n,(n-1)s + 1) = 1$, which simplifies to $\gcd(n,s-1) = 1$.
\end{proof}

%\noindent
%The case when $n$ is composite is similar, except there are multiple classes of values for $k$ and $\lambda$, each for a %distinct factor of $n$ greater than 1.

%We will focus our attention on basic quadruples for the remainder of this note.

\subsection{Some examples of basic OAs}
\label{examples.sec}

An optimal  $\OA_{\lambda}(k,n)$ has
a total of $\lambda n^2 = m(k(n-1)+1)$ rows. If we delete $m$ rows of the form
$0 \; 0 \cdots 0$, then the number of remaining rows is 
$mk(n-1)$, which is divisible by $k$. This suggests that we might attempt to construct
the $\OA_{\lambda}(k,n)$ by  cyclically rotating  $m(n-1)$ ``starting rows'' $k$ times.

We illustrate the technique in a small example.

\begin{example} We construct a basic  $\OA_{3}(5,2)$ from the following two  starting rows:
\[ \begin{array}{ccccc}
0& 0& 1 & 1& 1\\
0& 1& 0 & 1& 1
\end{array}
\]

We cyclically rotate these starting rows five times, and then adjoin $m=2$ rows of $0$'s.
This yields the desired orthogonal array, which is presented in Figure \ref{fig1}.

\begin{figure}[htb]
\[ \begin{array}{ccccc}
0& 0& 0 & 0& 0\\
0& 0& 0 & 0& 0\\ \hline
0& 0& 1 & 1& 1\\
1& 0& 0 & 1& 1\\
1& 1& 0 & 0& 1\\
1& 1& 1 & 0& 0\\
0& 1& 1 & 1& 0\\ \hline
0& 1& 0 & 1& 1\\
1& 0& 1 & 0& 1\\
1& 1& 0 & 1& 0\\
0& 1& 1 & 0& 1\\
1& 0& 1 & 1& 0
\end{array}
\]
\caption{A basic  $\OA_{3}(5,2)$}
\label{fig1}
\end{figure}

It is possible to verify that this process will yield an $\OA_{3}(5,2)$ 
without actually constructing the
whole array. It suffices to look all the ordered pairs that are (cyclically) a fixed distance 
apart in the starting rows. Each row has five entries, so we only have to consider pairs at 
distance one and two, because $\lfloor \frac{5}{2} \rfloor = 2$.
\begin{itemize}
\item For distance one, we have $00, 01,11,11,10$ from the first starting row and 
$01, 10,01,11,10$ from the second starting row. We see we have three occurrences of each of $01,10$ and $11$, and
one occurrence of $00$.
 \item For distance two, we have $01, 01,11,10,10$ from the first starting row and 
$00, 11,01,10,11$ from the second starting row. Again, we have three occurrences of each of $01,10$ and $11$, and
one occurrence of $00$.
\end{itemize}
This means that, when we rotate the starting rows and adjoin two rows of $0$'s, we are guaranteed to get
the desired orthogonal array.

Finally, we note that existence of a basic  $\OA_{3}(5,2)$ is also reported in \cite[Table 4]{CT}.

\end{example}

We now give starting rows for a few other small examples.

\begin{example} 
A basic  $\OA_{5}(9,2)$ with $m=2$ can be constructed from the following two  starting rows:
\[ \begin{array}{ccccccccc}
0& 0& 0 & 1& 0& 1& 1 & 1& 1\\
0& 0& 1 & 1& 0& 1& 0 & 1& 1
\end{array}
\]
Existence of a basic  $\OA_{5}(9,2)$ is also reported in \cite[Table 4]{CT}.
\end{example}

\begin{example} 
\label{OA7,13,2}
A basic  $\OA_{7}(13,2)$  with $m=2$ can be constructed from the following two  starting rows:
\[ \begin{array}{ccccccccccccc}
0& 0& 0 & 0& 1& 0& 1 & 1& 1& 0& 1 & 1& 1\\
0& 0& 1 & 0& 1& 1& 0 & 0& 1& 1& 1 & 0& 1
\end{array}
\]
\end{example}

\begin{example} 
A basic  $\OA_{9}(17,2)$  with $m=2$ can be constructed from the following two  starting rows:
\[ \begin{array}{ccccccccccccccccc}
0& 0& 0 & 0& 0& 1& 0 & 1& 1& 1& 0 & 1& 1& 0& 1 & 1& 1\\
0& 0& 0 & 1& 0& 1& 1 & 0& 1& 0& 1 & 1& 0& 0& 1 & 1& 1
\end{array}
\]
\end{example}

\begin{example} 
\label{ex5}
%A basic  $\OA_{5}(7,3)$  with $m=3$ can be constructed from the following six  starting rows:
%\[ \begin{array}{ccccccc}
%0 & 0 & 1 & 1 & 1 & 1 & 1 \\
%0 & 0 & 1 & 2 & 1 & 2 & 2 \\
%0 & 1 & 0 & 2 & 2 & 2 & 1 \\
%0 & 1 & 2 & 0 & 2 & 1 & 2 \\
%0 & 1 & 2 & 2 & 1 & 0 & 2 \\
%0 & 2 & 1 & 1 & 0 & 2 & 2 
%\end{array}
%\]
We use a slightly different technique to obtain a  basic  $\OA_{5}(7,3)$  with $m=3$.
We  have three  starting rows, consisting of symbols from the set $\{\infty\} \cup \zed_2$:
\[ \begin{array}{ccccccc}
\infty & \infty & 0 & 0 & 0 & 0 & 1\\
\infty & 1 & \infty & 0 & 1 & 1 & 0\\
\infty & 1 & 1 & \infty & 1 & 0 & 1
\end{array}
\]
First, cyclically rotate each starting row seven times. Then develop each row 
modulo $2$ (the point $\infty$ is fixed). Finally, adjoin three rows of $\infty$'s.
The resulting $3 \times 7 \times 2 + 3 = 45$ rows form a basic $\OA_{5}(7,3)$.

Existence of a basic  $\OA_{5}(7,3)$ is also reported in \cite[Table 4]{CT}.
\end{example}

\begin{example} 
\label{ex794}
We construct a  basic  $\OA_{7}(9,4)$  with $m=4$ from
four  starting rows, consisting of symbols from the set $\{\infty\} \cup \zed_3$:
\[ \begin{array}{ccccccccc}
 \infty & \infty & 0 & 0 & 0 & 0 & 1 & 0 & 2 \\
 \infty & 0 & \infty & 1 & 0 & 0 & 1 & 2 & 1 \\
 \infty & 0 & 1 & \infty & 1 & 1 & 0 & 1 & 2 \\
 \infty & 0 & 0 & 2 & \infty & 2 & 1 & 1 & 2
\end{array}
\]
First, cyclically rotate each starting row nine times. Then develop each row 
modulo $3$ (the point $\infty$ is fixed). Finally, adjoin four rows of $\infty$'s.
The resulting $4 \times 9 \times 3 + 4 = 112$ rows form a basic $\OA_{7}(9,4)$.
\end{example}

\section{Basic OAs for $n=2$}
\label{basic.sec}

In this section, we determine all the basic quadruples $(m,\lambda,k,2)$ with $m > 1$.
For these quadruples, we can construct a basic $\OA_{\lambda}(k,2)$ 
provided that a suitable Hadamard matrix exists.

%\begin{lemma}
%\label{L2.1} If $(m,\lambda,k,2)$ is a feasible quadruple, then $k$ is odd.
%\begin{proof}
%Since $n = 2$, we have $m = 4 \lambda / (k+1)$.
%We require that \[\overline{a} = \frac{k(\lambda n - m)}{\lambda n^2-m}\] is an integer. 
%Substituting $n = 2$ and $m = 4 \lambda / (k+1)$ into this formula, we see that
%$\overline{a} = (k-1)/2$, so $k$ is odd.
%\end{proof}
%\end{lemma}

\begin{lemma}
\label{basicn=2} If $(m,\lambda,k,2)$ is a basic quadruple with $m > 1$, then
$m=2$. Further, $\lambda = 2t+1$ and $k = 4t+1$ for some positive integer $t$.
\end{lemma}

\begin{proof}
Take $n=2$ in Theorem \ref{prime.thm}. Then $m=2$ and
we have $k=2s+1$ and $\lambda = s+1$, where $\gcd(2, s-1) = 1$.
Hence $s$ is even. Writing $s=2t$, we obtain $\lambda = 2t+1$ and $k = 4t+1$.
\end{proof}

\begin{theorem}
\label{BIBDtoOA.thm}
There exists a basic $\OA_{2t+1}(4t+1,2)$ 
if and only if there is a $(4t+1,2t+1,2t+1)$-BIBD.
\end{theorem}

\begin{proof} First suppose that a $(4t+1,2t+1,2t+1)$-BIBD exists. This BIBD  has
$b = 8t+2$ blocks and replication number $r = 4t+2$. Let $M$ be the 
$b$ by $v$ incidence matrix of this BIBD. Construct the matrix
\[ A = \left( \begin{array}{c} 
0\; 0 \cdots 0\\
0\; 0 \cdots 0\\
M
\end{array} \right) . \]
We claim that $A$ is a basic $\OA_{2t+1}(4t+1,2)$.
Clearly the first two rows are identical, so $m=2$ and we just need to verify that $A$ is an OA with the
stated parameters.

Choose any two distinct columns of $A$. These columns correspond to two points in the BIBD, say $x$ and $y$,
where $x \neq y$.
The number of occurrences of $1\; 1$ in these two columns is $\lambda = 2t+1$.
The number of occurrences of $0\; 1$ is $r - \lambda = 2t+1$, as is the number of
occurrences of $1\; 0$. Finally, the number of occurrences of $0\; 0$ is
$2 + 8t+2 - 3(2t+1) = 2t+1$.

Conversely, suppose $A$ is  a basic  $\OA_{2t+1}(4t+1,2)$.
We can assume that the symbols in the OA are $0$ and $1$. 
Without loss of generality, suppose that there are $m=2$ rows consisting entirely of zeroes, and
then delete them, creating a $8t+2$ by $4t+1$ matrix $M$. We will show that $M$ is
the incidence matrix of a $(4t+1,2t+1,2t+1)$-BIBD. 

Clearly $M$ is the incidence matrix of a set system on $4t+1$ points.
Given any two points $x$ and $y$, the number of blocks containing these two points
is the same as the number of occurrences of $1\; 1$ in the two associated columns of the OA,
which is $2t+1$. To complete the proof, we show that every block in the set system
has size $2t+1$. This can be seen easily by recalling from Lemma \ref{abar.lem}  that every row of $A$ contains
exactly $\overline{a}  = (4t+1 -1)/2 = 2t$ zeroes. 
Therefore, the number of ones in a row of $A$ is $4t+1 - \overline{a} = 2t+1$. 
This completes the proof.
\end{proof}

Now we show how to construct a basic $\OA_{2t+1}(4t+1,2)$ from a certain Hadamard matrix.
We use a few standard results, all of which can be found in \cite{Stinson}, for example.

\begin{theorem}
If there exists a Hadamard matrix of order $8t+4$, then there exists a basic $\OA_{2t+1}(4t+1,2)$.
% $(4t+1,2t+1,2t+1)$-BIBD.
\end{theorem}

\begin{proof}
It is well-known that a Hadamard matrix of order $8t+4$ is equivalent to a 
symmetric $(8t+3,4t+1,2t)$-BIBD. The derived BIBD is a
$(4t+1,2t,2t-1)$-BIBD. If we then complement every block in this BIBD, we obtain
a $(4t+1,2t+1,2t+1)$-BIBD. Finally, apply Theorem \ref{BIBDtoOA.thm}.
\end{proof}

It is known that Hadamard matrices exist for all orders $n \equiv 0 \bmod 4$, 
$4 \leq n < 668$, and it is conjectured that Hadamard matrices exist for all orders $n \equiv 0 \bmod 4$, $n \geq 4$.

\section{General constructions for optimal OAs}
\label{general.sec}

Suppose we fix $k$ and $n$, where $k \geq n+1$.
Denote $\rho = n^2 / (k(n-1)+1)$; then $\rho \leq 1$.
Our goal is to find an optimal $\OA_{\lambda}(k,n)$ for some value of $\lambda$.
Note that $m / \lambda = \rho$ in such an OA.
Also,  $\overline{a} = (k-1)/n \geq 1$.

One possible approach would be to take all possible $k$-tuples that contain precisely
$\overline{a} = (k-1)/n$ occurrences of $0$, and then adjoin an appropriate number of
rows consisting entirely of $0$'s.

We illustrate the idea using a small example.

\begin{example}
\label{ex6}
Suppose we take  $k=7$ and $n=3$.
Here we have $\rho = 3/5$ and $\overline{a} = 2$.
There are \[\binom{7}{2} \times 2^5 = 21 \times 32 = 672\] 
$7$-tuples on the symbol set $\{0,1,2\}$ that
contain precisely two zeroes. If we then 
adjoin $48$ rows of $0$'s, it is not hard to check that
we obtain an optimal $\OA_{80}(7,3)$. The ratio $\rho = m / \lambda = 48/80 = 3/5$,
as required.

Of course, we know from Example \ref{ex5} that a basic $\OA_{5}(7,3)$ exists.
The value of $\lambda$ in the above-constructed optimal $\OA_{80}(7,3)$ is much larger.
\end{example}

As in Example \ref{ex6}, we take all possible $k$-tuples that contain precisely $\overline{a} = (k-1)/n$ occurrences of 0. In order to have an orthogonal array, any two columns must contain every ordered pair of symbols exactly $\lambda$ times, for some $\lambda$. Since every possible $k$-tuple containing $\overline{a}$ $0$'s
is used, it suffices to consider the first two columns. If these two columns contain every ordered pair the same number of times, then so will every other pair of columns in the array. 

There are four cases to consider for the first two elements in a  row: $00$, $0x$, $x0$, and $xy$, where $x, y \in \{1,\dots,n-1\}$. We consider  these cases in turn.

For each $x \in \{1, \dots, n-1\}$, there are
\begin{equation}
\label{0x.eq}
\binom{k-2}{\overline{a}-1}(n-1)^{k-\overline{a}-1} 
\end{equation}
rows beginning with $0x$. This result is the same for rows beginning with $x0$. For each $x,y \in \{1, \dots, n-1\}$, there are
\begin{equation}
\label{xy.eq}
\binom{k-2}{\overline{a}}(n-1)^{k-\overline{a}-2} 
\end{equation}
rows beginning with $xy$. Since $\overline{a} = (k-1)/n$, we find that
\begin{equation}
\label{eq6}
\frac{k-\overline{a}-1}{n-1} = \frac{k-1 - \frac{k-1}{n}}{n-1} 
= \frac{k-1}{n} 
= \overline{a}.
\end{equation}
Using (\ref{eq6}), it is now easy to show that  (\ref{0x.eq}) and (\ref{xy.eq}) are equal:
\begin{eqnarray*}
\binom{k-2}{\overline{a}}(n-1)^{k-\overline{a}-2} &=& 
\frac{k - \overline{a} - 1}{\overline{a}}
\binom{k-2}{\overline{a}-1} \frac{(n-1)^{k-\overline{a}-1}}{n-1} \\
& = & \binom{k-2}{\overline{a}-1}(n-1)^{k-\overline{a}-1}.
\end{eqnarray*}

Therefore, each ordered pair other than $00$ appears the same number of times
(given by (\ref{0x.eq}) or (\ref{xy.eq})), which we denote by
$\lambda$. Thus, adjoining the appropriate number of rows consisting entirely of $0$'s will result in an orthogonal array.

The number of rows beginning with $00$ is  
\[
{k-2 \choose \overline{a}-2}(n-1)^{k-\overline{a}}.
\]
It then follows that we need to adjoin 
\[m = \lambda - {k-2 \choose \overline{a}-2}(n-1)^{k-\overline{a}}\]
rows of $0$'s.
Substituting the value of $\lambda$ obtained from (\ref{0x.eq}), we have
\begin{eqnarray*}
m &=& \binom{k-2}{\overline{a}-1}(n-1)^{k-\overline{a}-1} 
- {k-2 \choose \overline{a}-2}(n-1)^{k-\overline{a}}\\
&=& \binom{k-2}{\overline{a}-1}(n-1)^{k-\overline{a}-1} -
\frac{\overline{a}-1}{k - \overline{a}}\binom{k-2}{\overline{a}-1}(n-1)(n-1)^{k-\overline{a}-1} \\
&=& \binom{k-2}{\overline{a}-1}(n-1)^{k-\overline{a}-1} \left( 1 - \frac{(n-1)(\overline{a}-1)}{k - \overline{a}} \right)\\
& = & \lambda \left( 1 - \frac{(n-1)(\overline{a}-1)}{k - \overline{a}} \right).
\end{eqnarray*}
Therefore, substituting $\overline{a} = (k-1)/n$, we have
\begin{eqnarray*}
\frac{m}{\lambda} &=&  1 - \frac{(n-1)(\overline{a}-1)}{k - \overline{a}}\\
& = & 1 - \frac{(n-1)(\frac{k-n-1}{n})}{k - \frac{k-1}{n}}\\
&=& 1 - \frac{(n-1)(k-n-1)}{kn-k+1}\\
&=& \frac{n^2}{k(n-1)+1}\\
&=& \rho,
\end{eqnarray*}
as desired.

Thus we have proven the following result.

\begin{theorem}
\label{general.thm}
Suppose $n^2 \leq k(n-1)+1$ and suppose $\overline{a} = (k-1)/n$ is an integer. 
Then there is an optimal
$\OA_{\lambda}(k,n)$, where
\begin{equation}
\label{general.eq}
 \lambda = \binom{k-2}{\overline{a}-1}(n-1)^{k-\overline{a}-1}.
 \end{equation}
 \end{theorem}
 
 \subsection{An improvement}

We next show that the $\OA_{\lambda}(k,n)$ constructed in Theorem \ref{general.thm}
can be partitioned into $n-1$ optimal  
$\OA_{\lambda/(n-1)}(k,n)$. In order to describe how this is done, it is useful to
change the set of symbols on which the OAs are defined. Suppose we begin with the above-mentioned
$\OA_{\lambda}(k,n)$, constructed on symbols $0, \dots , n-1$. Delete the 
$m = \lambda n^2 / (k(n-1)+1)$ rows
of $0$'s.
Then we replace all occurrences of $0$ in the remaining rows by $\infty$, 
and the symbols $1, \dots , n-1$ are replaced
by $0, \dots , n-2$, respectively. We consider the new symbol set as $\{\infty\} \cup \zed_{n-1}$. 
Denote the resulting array by $A$.

For any row $\mathbf{r}$ of $A$, let $s(\mathbf{r})$ denote the sum of the non-infinite
elements in row $\mathbf{r}$, reduced modulo $n-1$. Then, for any $i \in \{ 0,\dots, n-2\}$,
let $A_i$ consist of all the rows $\mathbf{r}$ of $A$ such that $s(\mathbf{r}) = i$.
Clearly every row of $A$ is in precisely one of $A_0, \dots , A_{n-2}$. 

It is obvious that every $A_i$ is fixed by any permutation of the columns. Therefore,
the number of occurrences of a particular pair of symbols in two given columns does 
not depend on the two columns that are chosen. Hence, we can restrict our attention
to the first two columns of these arrays.

For two (not necessarily distinct) symbols 
$x,y \in \{\infty\} \cup \zed_{n-1}$ and for $0 \leq i \leq n-2$,
let $\lambda_i(x,y)$ denote the number of occurrences of the ordered pair $(x,y)$ in the
first two columns of $A_i$.  Also, let
$\lambda(x,y)$ denote the number of occurrences of $(x,y)$ in the
first two columns of $A$. Therefore
\[ \lambda(x,y) = 
\begin{cases} \lambda & \text{if $(x,y) \neq (\infty,\infty)$} \\
\lambda - m & \text{if $(x,y) =(\infty,\infty)$,}
\end{cases}
\]
where \[ \lambda = \binom{k-2}{\overline{a}-1}(n-1)^{k-\overline{a}-1}\]
and 
\[m = \frac{\lambda n^2 }{ k(n-1)+1}.\]

We will show, for any $x,y$, that $\lambda_i(x,y)$  is independent of $i$.
First however, we state and prove a small technical lemma.

\begin{lemma}
\label{tech.lem}
Suppose that $n \geq 2$, $\overline{a} = (k-1)/n$ is an integer and $\overline{a} \geq 1$.
Then $\overline{a} +2 \leq k-1$ unless $n=2$, $k=3$ and $\overline{a} = 1$.
\end{lemma}

\begin{proof}
First, suppose $\overline{a} \geq 2$, so $(k-1)/n \geq 2$.
The following inequalities are equivalent:
\begin{eqnarray*}
2 + \overline{a} & \leq & k-1\\
2 + \frac{k-1}{n} & \leq & k-1\\
\frac{(k-1)(n-1)}{n} & \geq & 2.
\end{eqnarray*}
However, 
\[ \frac{(k-1)(n-1)}{n} \geq 2(n-1) \geq 2 \]
because $n \geq 2$.

Now, suppose $\overline{a} =1$ and $n \geq 3$, We have $k = n+1$, so it follows that 
$\overline{a} + 2 = 3 \leq n = k-1$.
\end{proof}

\begin{remark}
\label{1.rem}
The exception to Lemma \ref{tech.lem} is $n=2$, $k=3$ and $\overline{a} = 1$.
But this is a trivial case, as a basic $\OA_{1}(3,2)$ exists, and this OA can
trivially be ``partitioned'' into $n-1 = 1$ OAs.
\end{remark}

In the following discussion, we assume that $n \geq 2$, $\overline{a} = (k-1)/n$ 
is an integer, $\overline{a} \geq 1$,
and $(n,k) \neq (2,3)$.

We next define a mapping $f$ on the rows of $A$.
Let $\mathbf{r}$ be any row of $A$, where $\mathbf{r} = (x_1, \dots , x_k)$.
Let $j_0 = \min \{ j : 3 \leq j \leq k, x_j \neq \infty\}$
(there are $\overline{a}$ occurrences of $\infty$ in $\mathbf{r}$, so 
Lemma \ref{tech.lem} ensures that $\{ j : 3 \leq j \leq k, x_j \neq \infty\} \neq \emptyset$
and hence $j_0$ exists).

Let $\kappa \in \zed_{n-1}$. Define $f(\mathbf{r}) = (y_1, \dots, y_k)$, where
\[ y_j = \begin{cases}
x_j + \kappa \bmod (n-1)& \text{if $j = j_0$,}\\
x_j & \text{otherwise},
\end{cases}
\]
for $j = 1, \dots , k$.

The  process above can also be described as follows. Find the first entry in row $\mathbf{r}$, 
past the second column, that is not equal to $\infty$. Then add $\kappa$ modulo $n-1$ to that entry.

It is clear that the mapping $f$ gives a bijection from the rows in $A_i$ to the rows in
$A_{i+\kappa \bmod (n-1)}$,
for $i \in \zed_{n-1}$. Also,  $f$ leaves the points in the first two columns of any row of $A$ 
unaltered. Since
\[ \sum_{i \in \zed_{n-1}} \lambda_i(x,y) = \lambda(x,y),\]
we have $\lambda_i(x,y) = \lambda(x,y)/(n-1)$ for all $i \in \zed_{n-1}$. Therefore,
if we adjoin $m/(n-1)$ rows of $\infty$'s to any $A_i$, we obtain an optimal 
$\OA_{\lambda}(k,n)$, where
\[ \lambda = \binom{k-2}{\overline{a}-1}(n-1)^{k-\overline{a}-2}.\]

The above discussion, along with Remark \ref{1.rem}, proves the following.

\begin{theorem}
\label{partition.thm}
Suppose $n^2 \leq k(n-1)+1$ and suppose $\overline{a} = (k-1)/n \geq 1$ is an integer. 
Then there is an optimal
$\OA_{\lambda}(k,n)$, where
\[ \lambda = \binom{k-2}{\overline{a}-1}(n-1)^{k-\overline{a}-1},\]
that can be partitioned into $n-1$ optimal $\OA_{\lambda/(n-1)}(k,n)$.
\end{theorem}

\begin{corollary}
\label{improved.thm}
Suppose $n^2 \leq  k(n-1)+1$ and suppose $\overline{a} = (k-1)/n \geq 1$ is an integer. 
Then there is an optimal
$\OA_{\lambda}(k,n)$, where
\[ \lambda = \binom{k-2}{\overline{a}-1}(n-1)^{k-\overline{a}-2}.\]
\end{corollary}

\begin{remark}
We have noted that $\overline{a} = 1$ when $k = n+1$.
A basic OA with these parameters is in fact an $\OA_{1}(n+1,n)$, which is 
equivalent to a projective plane of order $n$. %Such an OA does not contain
%any ``repeated'' rows, as $m=1$. 
However, in cases where a projective plane of order $n$
is known not to exist, Corollary \ref{improved.thm} provides examples of 
$\OA_{\lambda}(n+1,n)$ for certain large values of $\lambda$.
\end{remark}

\subsection{A further improvement}

We now prove an extension of Theorem \ref{partition.thm} where we can sometimes reduce the
value of $\lambda$ by additional factors of $n-1$, depending on the parameters $k$ and $n$.

Suppose we can write $k = k_1+ \dots + k_{\gamma}$, where $k_1, \dots , k_{\gamma}$ are integers
such that $k_i \geq \overline{a}+3$ for $1 \leq i \leq \gamma$.
Evidently, we can take \[\gamma = \left\lfloor \frac{k}{\overline{a}+3}\right\rfloor.\]
As before, our starting point is the optimal $\OA_{\lambda}(k,n)$ obtained from
Theorem \ref{general.thm} in which the symbol set is $\{\infty\} \cup \zed_{n-1}$.
The value of $\lambda$ is given by (\ref{general.eq})
and there are $m$ rows consisting only of $\infty$'s. Delete these $m$ rows
and call the resulting array $A$.

Now, we consider the columns of $A$ to be partitioned into $\gamma$ \emph{classes}, 
where the $i$th class consists of $k_i$ columns.  Denote these classes as
$\C_1,  \dots , \C_{\gamma}$.

For any row $\mathbf{r}$ of $A$, let $s(\mathbf{r}) = (s_1, \dots , s_{\gamma})$,
where, for $1 \leq i \leq \gamma$, $s_i$ is the modulo $n-1$ sum of the non-infinite elements in 
row  $\mathbf{r}$ in the 
columns in $\C_i$. We will say that the $\gamma$-tuple 
$s(\mathbf{r})$ is the \emph{type} of row $\mathbf{r}$.

Then, for any possible type $\tau \in (\zed_{n-1})^{\gamma}$,
let $A_{\tau}$ consist of all the rows $\mathbf{r}$ of $A$ such that $s(\mathbf{r}) = \tau$.
(That is, $A_{\tau}$ comprises all the rows of $A$ having type $\tau$.)
This yields a partition of the rows of $A$ into $(n-1)^{\gamma}$ subsets.
We will show that each $A_i$, when augmented with an appropriate number of
rows of $\infty$'s, is an optimal $\OA_{\lambda}(k,n)$, where
\[ \lambda = \binom{k-2}{\overline{a}-1}(n-1)^{k-\overline{a}-\gamma -1}.
\] The number of rows of $\infty$'s to be added to each $A_{\tau}$ to construct
an orthogonal array is $m/(n-1)^{\gamma}$.

%\vspace{.1in}

We use a ``bijection'' proof similar to Theorem \ref{partition.thm}.
First, we note that any  permutation of the columns within any column class $\C_i$
is an automorphism of every $A_{\tau}$. Therefore, to prove that $A_{\tau}$
yields an orthogonal array (as described above), we just need to consider the ordered
pairs of symbols occurring in pairs
of columns of the following types:
\begin{description}
\item[(a)] the first two columns of any $\C_i$, and
\item[(b)] the first column of $\C_i$ and the the first column of $\C_j$, where $i \neq j$.
\end{description}

We can use a bijection similar to  Theorem \ref{partition.thm}, but we apply  the bijection
independently for every column class $\C_i$. 
Let $\mathbf{r}$ be any row of $A$, where $\mathbf{r} = (x_1, \dots , x_k)$.
For each column class $\C_i$, let $j_i$ be the first column within $\C_i$, 
past the second column, that is not equal to $\infty$ (note that $j_i$ exists because
$k_i \geq \overline{a}+3$).

Fix any $\gamma$-tuple  $\kappa \in (\zed_{n-1})^{\gamma}$.
Define the mapping $f$ as follows: $f(\mathbf{r}) = (y_1, \dots, y_k)$, where
\[ y_j = \begin{cases}
x_j + \kappa_i \bmod (n-1)& \text{if $j = j_i$ for some $i$,}\\
x_j & \text{otherwise},
\end{cases}
\]
for $j = 1, \dots , k$.

The mapping $f$ is a bijection from the rows in $A_{\tau}$ to the rows in
$A_{\tau + \kappa \bmod (n-1)}$,
for any  $\tau \in (\zed_{n-1})^{\gamma}$. 
Also, for any row $\mathbf{r}$, $f$ leaves the points in the first two columns of every $\C_i$
unaltered. Therefore,  the ordered pairs occurring in pairs of columns 
of types {\bf (a)} and {\bf (b)} in any $A_{\tau}$ is independent of $\tau$. 
Hence, we obtain the following theorem and corollary.

\begin{theorem}
\label{partition2.thm}
Suppose $n^2 \leq k(n-1)+1$ and suppose $\overline{a} = (k-1)/n$ is an integer. 
Suppose \[\gamma =   \left\lfloor \frac{k}{\overline{a}+3}\right\rfloor \geq 1.\]
Then there is an optimal
$\OA_{\lambda}(k,n)$, where
\[ \lambda = \binom{k-2}{\overline{a}-1}(n-1)^{k-\overline{a}-1},\]
that can be partitioned into $(n-1)^{\gamma}$ optimal $\OA_{\lambda/(n-1)^{\gamma}}(k,n)$.
\end{theorem}

\begin{corollary}
\label{improved2.thm}
Suppose $n^2 \leq  k(n-1)+1$ and suppose $\overline{a} = (k-1)/n \geq 1$ is an integer. 
Suppose \[\gamma =   \left\lfloor \frac{k}{\overline{a}+3}\right\rfloor \geq 1.\]
Then there is an optimal
$\OA_{\lambda}(k,n)$, where
\[ \lambda = \binom{k-2}{\overline{a}-1}(n-1)^{k-\overline{a}-\gamma - 1}.\]
\end{corollary}

\begin{example}
Suppose we take $k = 16$ and $n=3$. Then $\overline{a} = 5$ and $\gamma = 2$.
From Corollary \ref{improved2.thm}, we obtain an optimal $\OA_{\lambda}(k,n)$,
where \[ \lambda = \binom{14}{4}2^{8}.\]
\end{example}

\section{Deleting columns from optimal OAs}
\label{delete.sec}

An optimal $\OA_{\lambda}(k,n)$ can exist only when $k \equiv 1 \bmod n$. However,
it is certainly of interest to determine the maximum possible ratio 
$m/\lambda$ for values of $k$ and $n$ where $k \not\equiv 1 \bmod n$. 
We might reasonably expect that deleting a small number of columns, say $s$ columns, 
from an optimal $\OA_{\lambda}(k,n)$
could yield an $\OA_{\lambda}(k-s,n)$ where the ratio $m/\lambda$ is as large as possible.

\subsection{Deleting a single column}

We first prove that this approach works for $s=1$ by proving 
a modification of Theorem \ref{repeated.thm}. The modified bound 
uses a similar proof technique to \cite[Theorem 3.1]{St82}
(see also \cite[Theorem 8.7]{Stinson}).
%\cite[Theorem 2.4]{Sti18}. 

%Our modified bound is
%useful in studying the cases where $k \equiv 0 \bmod{n}$.

\begin{theorem}
\label{repeated2.thm}
Let $k \geq 2$, $n \geq 2$, $\lambda \geq 1$ be integers. 
Suppose $A$ is an $\OA_{\lambda}(k,n)$ containing a row $0 \; 0 \cdots 0$ that is repeated $m$ times, 
and let $\alpha$ be a positive integer.
Then
\begin{equation}
\label{modified.eq}
m \leq \frac{\lambda(k(k-1) - 2\alpha k n + (\alpha^2+\alpha)n^2)}{k(k-1)-2\alpha k+\alpha^2+\alpha}.
\end{equation}
Further, equality occurs if and only if every row has either $\alpha$ or $\alpha+1$ occurrences of the symbol 0.
\end{theorem}

\begin{proof}
Suppose the last $m$ rows of $A$ are $0 \; 0 \cdots 0$.
Let $a_i$ denote the number of occurrences of the symbol $0$ in row $i$ of $A$. Define
$N = \lambda n^2$. It is clear 
from elementary counting that 
\begin{eqnarray*}
\sum _{i=1}^{N-m} a_i &=& k(\lambda n - m) \quad \text{and}  \\
\sum _{i=1}^{N-m} a_i(a_i-1) &=& k(k-1)(\lambda  - m). % \quad \text{and} \\
%\sum _{i=1}^{N-m} {a_i}^2 &=& k(k(\lambda - m) + \lambda(n-1)) .
\end{eqnarray*}

Let $\alpha$ be a positive integer. Then
\begin{eqnarray*}
0 & \leq & \sum_{i=1}^{N-m} (a_i - \alpha)(a_i - \alpha - 1) \\
&=& \sum_{i=1}^{N-m} {a_i}(a_i-1) -2\alpha \sum_{i=0}^{N-m} a_i + (\alpha^2 + \alpha)(N-m) \\
&=& k(k-1)(\lambda-m) -2\alpha k(\lambda n - m) + (\alpha^2+\alpha)(\lambda n^2 -m).
\end{eqnarray*}
Solving for $m$, we obtain (\ref{modified.eq}).

In the case of equality, every term in the sum
\[ \sum_{i=1}^{N-m} (a_i - \alpha)(a_i - \alpha - 1) \] must equal $0$.
Therefore, $a_i \in \{\alpha,\alpha+1\}$ for every $i$.
\end{proof}

We examine a special case of the bound proven in Theorem \ref{repeated2.thm}.

\begin{corollary}
Let $k \geq 2$, $n \geq 2$, $\lambda \geq 1$ be integers. If $k \equiv 0 \bmod{n}$ and there is an $\OA_{\lambda}(k,n)$ containing a row  that is repeated $m$ times, then
\begin{equation}
\label{bound2.eq}
m \leq \frac{\lambda n^2}{k(n-1)+n}.
\end{equation}
\end{corollary}

\begin{proof}
Write $k = sn$ for some integer $s$ and take $\alpha =s-1$ in Theorem \ref{repeated2.thm}. We find that
\begin{eqnarray*}
m &\leq &\frac{\lambda (sn(sn-1) -2(s-1)sn^2 + ((s-1)^2+(s-1))n^2)}{sn(sn-1)-2(s-1)sn + (s-1)^2 + s-1} \\
&=& \frac{\lambda (s^2n^2-sn-2s^2n^2 +2sn^2 +s^2n^2 -2sn^2 +n^2 +sn^2-n^2)}{s^2n^2-sn-2s^2n+2sn+s^2-2s+1+s-1} \\
&=& \frac{\lambda (sn^2-sn)}{s^2n^2-2s^2n+sn+s^2-s} \\
&=& \frac{\lambda sn(n-1)}{s^2(n-1)^2 + s(n-1)} \\
&=& \frac{\lambda n}{s(n-1) + 1} \\
&=& \frac{\lambda n^2}{k(n-1)+n}.
\end{eqnarray*}
\end{proof}

\begin{theorem} Suppose there is an optimal $\OA_{\lambda}(k+1,n)$ 
containing a row that is repeated $m$ times. Then $k \equiv 0 \bmod n$ 
and there is
an $\OA_{\lambda}(k,n)$ 
containing a row that is repeated $m$ times, where
(\ref{bound2.eq}) is met with equality.
\end{theorem}
\begin{proof}
Let $A$ be an optimal $\OA_{\lambda}(k+1,n)$ 
containing a row $0 \; 0 \cdots 0$ that is repeated $m$ times. 
If we delete any one column from $A$, then the resulting array $A'$ is
an $\OA_{\lambda}(k,n)$ 
containing a row $0 \; 0 \cdots 0$ that is repeated $m$ times.
Since $A$ is optimal, we have \[m = \frac{\lambda n^2}{(k+1)(n-1)+1} = \frac{\lambda n^2}{k(n-1)+n}\] 
from (\ref{m.eq}), so (\ref{bound2.eq}) is met with equality. Further, 
$\overline{a} = k/n$ must be a positive integer from Corollary \ref{a.cor}. 
\end{proof}

\subsection{Deleting multiple columns}

We now consider a different approach. 
Suppose we return to our original bound, Theorem \ref{repeated.thm}.
Since $m$ must be an integer, the following variation is immediate. 

\begin{theorem}
\label{repeated3.thm}
Let $k \geq 2$, $n \geq 2$ and $\lambda \geq 1$ be integers.
If there is an $\OA_{\lambda}(k,n)$ containing a row that is repeated $m$ times, then
\[
m \leq \left\lfloor \frac{\lambda n^2}{k(n-1)+1} \right\rfloor.
\]
\end{theorem}
An orthogonal array for which the bound in Theorem 
\ref{repeated3.thm} is met with equality will be termed \emph{$m$-optimal}.

Assume we start with an optimal $\OA_{\lambda}(k,n)$, so 
\[
m =  \frac{\lambda n^2}{k(n-1)+1} .
\]
We determine how many 
columns we can remove, denoted by $s$, such that 
\begin{equation*}
m = \left\lfloor\frac{\lambda n^2}{(k-s)(n-1)+1}\right\rfloor.
\end{equation*}
The resulting  $\OA_{\lambda}(k,n)$ will be $m$-optimal.

The following numerical lemma will be useful.

\begin{lemma}
\label{s.ineq}
Suppose that \[
m =  \frac{\lambda n^2}{k(n-1)+1} 
\] is an integer 
and \begin{equation}
\label{delete.eq}
s < \frac{(k(n-1)+1)^2}{(n-1)(\lambda n^2 + k(n-1)+1)}
\end{equation}
is a positive integer. 
Then 
\begin{equation*}
m = \left\lfloor\frac{\lambda n^2}{(k-s)(n-1)+1}\right\rfloor .
\end{equation*}
\end{lemma}

\begin{proof}
From (\ref{delete.eq}), we find that \[s(n-1)(\lambda n^2 + k(n-1)+1) 
< (k(n-1)+1)^2.\] 
Rearranging this inequality, we see that 
\[ s\lambda n^2(n-1) < ((k-s)(n-1)+1)(k(n-1) +1),\] and therefore 
\begin{eqnarray*}
1 &> &\frac{\lambda n^2s(n-1)}{((k-s)(n-1)+1)(k(n-1) +1)} \\
%&= &\frac{\lambda n^2 (k(n-1)+1) - \lambda n^2(k(n-1)+1 - s(n-1)}{((k-s)(n-1)+1)(k(n-1)+1)} \\
&= &\lambda n^2 \left( \frac{1}{(k-s)(n-1)+1} - \frac{1}{k(n-1)+1}\right).
\end{eqnarray*}
Since $m$ is an integer, this proves that 
\begin{equation*}
m =  \frac{\lambda n^2}{k(n-1)+1} = \left\lfloor\frac{\lambda n^2}{(k-s)(n-1)+1}\right\rfloor.
\end{equation*}
\end{proof}
The following theorem is now an immediate consequence of Lemma \ref{s.ineq}.

\begin{theorem}
\label{deleted.thm}
If there is an optimal $\OA_{\lambda}(k,n)$ and $s$ is a positive integer such that
 (\ref{delete.eq}) holds, 
then there is an $m$-optimal $\OA_{\lambda}(k-s, n)$.\end{theorem}

\begin{proof}
Let $A$ be an optimal $\OA_{\lambda}(k, n)$ containing a row 
that is repeated 
$m = \lambda n^2/(k(n-1)+1)$ times. If we delete any $s$ columns from $A$, then the resulting array $A'$ is an $\OA_{\lambda}(k-s, n)$ containing a row that is repeated $m$ times. 

From Lemma \ref{s.ineq}, we have that
\begin{equation*}
m =   \left\lfloor\frac{\lambda n^2}{(k-s)(n-1)+1}\right\rfloor,
\end{equation*}
so the resulting $\OA_{\lambda}(k-s, n)$ is $m$-optimal.
\end{proof}

To illustrate the application of Theorem \ref{deleted.thm},  we consider the case  $n = 2$. 

\begin{theorem}
\label{delete.cor}
Suppose there is a basic $\OA_{\lambda}(k,2)$ 
and suppose
\begin{equation*}
s < \frac{k+1}{3}
\end{equation*}
is a positive integer. Then there is an $m$-optimal $\OA_{\lambda}(k-s,2)$. 
\end{theorem}

\begin{proof}
From Lemma \ref{basicn=2}, a basic $\OA_{\lambda}(k,2)$ has $\lambda = 2t+1$ and $k = 4t+1$, where $t$ is a positive integer, and $m=2$.
If we take $n = 2$, $k = 4t+1$, and $\lambda = 2t+1$ in (\ref{delete.eq}), 
the inequality becomes
\begin{eqnarray*}
s &<& \frac{(k(n-1)+1)^2}{(n-1)(\lambda n^2 + k(n-1)+1)}\\
&=& \frac{(k+1)^2}{4\lambda + k + 1} \\
&=& \frac{(4t+2)^2}{4(2t+1)+(4t+1)+1} \\
%&=& \frac{(4t+2)^2}{3(4t + 2)} \\
&=& \frac{4t+2}{3}\\
&=& \frac{k+1}{3}.
\end{eqnarray*}
The stated result now follows from Theorem \ref{deleted.thm}.
\end{proof}

\begin{corollary}
For $1 \leq s \leq 4$, there is an $m$-optimal $\OA_7(13-s,2)$.
\end{corollary}
\begin{proof}
Begin with a basic $\OA_7(13,2)$ (Example \ref{OA7,13,2}) and take
$k=13$ in Theorem \ref{delete.cor}.
\end{proof}

We now discuss some examples of $m$-optimal OAs from Culus and Toulouse \cite{CT}.
\begin{example}
\label{mopt.ex}
Some OAs reported in \cite[Table 4]{CT} are in fact $m$-optimal:
\begin{itemize}
\item an $\OA_3(4,2)$ with $m=2$
\item an $\OA_5(8,2)$ with $m=2$
\item an $\OA_5(6,3)$ with $m=3$.
\end{itemize}
All of these OAs can be obtained by deleting a column from a basic OA.

Finally, the $\OA_3(5,3)$ with $m=2$ that is depicted in \cite[Table 2]{CT} is also $m$-optimal.
\end{example}

We make a few observations about the inequality in Theorem \ref{repeated3.thm}.

In Section \ref{prelim.sec}, we noted that taking multiple copies of an optimal OA would yield another
optimal OA. However, a similar result does not hold for $m$-optimal OAs. To illustrate, we consider 
OAs with $k=5$ and $n=3$. For an $\OA_{\lambda}(5,3)$,  Theorem \ref{repeated3.thm} asserts that
\[ m \leq \left\lfloor \frac{9 \lambda}{11} \right\rfloor.\]
The $\OA_3(5,3)$ with $m=2$ mentioned in Example \ref{mopt.ex} is $m$-optimal because
$\lfloor \frac{27}{11} \rfloor = 2$. If we take two copies of this OA, we obtain
an $\OA_6(5,3)$ with $m=4$. This OA is   $m$-optimal because $\lfloor \frac{54}{11} \rfloor = 4$.
However, if we take three copies of the $\OA_{3}(5,3)$, we obtain
an $\OA_9(5,3)$ with $m=6$, which is not $m$-optimal because $\lfloor \frac{81}{11} \rfloor = 7 > 6$.
(We do not know if an $\OA_9(5,3)$ with $m=7$ exists.)

Finally, consider a hypothetical $OA_{11}(5,3)$. Here, Theorem \ref{repeated3.thm} yields
$m \leq  9$.  An $OA_{11}(5,3)$ with $m=9$ would be optimal, because (\ref{m.eq}) holds.
However, an optimal $OA_{11}(5,3)$ cannot exist by Corollary \ref{a.cor}, because
$5 \not\equiv 1 \bmod 3$. Therefore there is no $m$-optimal $OA_{11}(5,3)$.

\section{Summary}
\label{summary.sec}

Clearly there is much work to be done on the problem of constructing optimal and basic
orthogonal arrays. At present, we only have a couple of examples of basic OAs with $n > 2$.
Thus, it is of particular interest to construct additional examples, or better yet, infinite
classes of these arrays.

\section*{Acknowledgements}

This work benefitted from the use of the CrySP RIPPLE Facility at the University of Waterloo.

\end{document}